\documentclass[10pt,reqno]{amsart}
\usepackage{}
\usepackage{bbm}
\usepackage{amsfonts}
\usepackage{amssymb}
\usepackage{amsmath}
\usepackage{leftidx}
\usepackage{extarrows}
\usepackage[all]{xy}
\usepackage{cases}
\usepackage{txfonts}
\usepackage{amsfonts}
\usepackage{mathrsfs}
\usepackage{amssymb}
\usepackage{amsmath}
\usepackage{epic}
\usepackage[top=1in, bottom=1in, left=1.25in, right=1.25in]{geometry}
\newtheorem{proposition}{Proposition}[section]
\newtheorem{lemma}[proposition]{Lemma}
\newtheorem{corollary}[proposition]{Corollary}

\theoremstyle{definition}
\newtheorem{definition}[proposition]{Definition}
\newtheorem{example}[proposition]{Example}

\theoremstyle{remark}
\newtheorem{remark}[proposition]{Remark}

\newcommand{\selabel}[1]{\label{se:#1}}
\newcommand{\seref}[1]{Section~\ref{se:#1}}
\newcommand{\lelabel}[1]{\label{le:#1}}
\newcommand{\leref}[1]{Lemma~\ref{le:#1}}
\newcommand{\prlabel}[1]{\label{pr:#1}}
\newcommand{\prref}[1]{Proposition~\ref{pr:#1}}
\newcommand{\colabel}[1]{\label{co:#1}}
\newcommand{\coref}[1]{Corollary~\ref{co:#1}}

\newcommand{\delabel}[1]{\label{de:#1}}

\newcommand{\eqlabel}[1]{\label{eq:#1}}
\newcommand{\equref}[1]{(\ref{eq:#1})}

\def\a{\alpha}

\def\b{\beta}

\def\D{\Delta}

\def\ep{\varepsilon}

\def\Ker{\mathrm{Ker}}

\def\op{\oplus}
\def\ot{\otimes}

\def\ra{\rightarrow}

\def\ss{\subseteq}

\def\Z{\mathbb{Z}}

\def\<{\leq}
\def\>{\geq}

\date{}
\begin{document}
\title{Dorroh extensions of algebras and coalgebras, II}
\thanks{}
\author{Lan You}
\address{School of Mathematical Science, Yangzhou University, Yangzhou 225002, China;
School of Mathematics and Physics, Yancheng Institute
of Technology, Yancheng 224051, China}
\email{youl@ycit.cn}
\author{Hui-Xiang Chen}
\address{School of Mathematical Science, Yangzhou University, Yangzhou 225002, China}
\email{hxchen@yzu.edu.cn}

\begin{abstract}
In this paper, we study Dorroh extensions of bialgebras and Hopf algebras. Let $(H,I)$ be both a Dorroh pair of algebras and a Dorroh pair of coalgebras. We give necessary and sufficient conditions for $H\ltimes_dI$ to be a bialgebra and a Hopf algebra, respectively. We also describe all ideals of Dorroh extensions of algebras and subcoalgebras of Dorroh extensions of coalgebras and compute these ideals and subcoalgebras for some concrete examples.
\end{abstract}
\maketitle

\section*{Introduction}
Dorroh \cite{D} gave a general way to  embed a ring $I$ without identity into a ring with an identity $\Z \op I$,
which is now called a Dorroh extension of $I$. In classical ring theory, Dorroh extension has become
an important method of constructing new rings and of analyzing properties of rings.
Many rings, such as trivial extension of a ring with a bimodule, triangular matrix ring, $\mathbb{N}$-graded ring, can be regarded as Dorroh extensions of rings. There are many papers to study Dorroh extensions of algebras, see \cite{AJM, CS, CMRS, DF, F, M, Ro}.

In \cite{YC}, we investigated Dorroh extensions of algebras, Dorroh extensions of coalgebras
and the duality between them. In this paper, we continue the study. The paper is organized as follows.
In \seref{sec1}, we study Dorroh extensions of bialgebras and Hopf algebras.
Let $(H,I)$ be meanwhile a Dorroh pair of algebras and a Dorroh pair of coalgebras.
Then $H\ltimes_dI$ is both an algebra and a coalgebra. We give necessary and sufficient conditions
for $H\ltimes_d I$ to be a bialgebra and a Hopf algebra, respectively.
In \seref{sec2}, we investigate the ideals of Dorroh extensions of algebras.
For a given algebra Dorroh extension $A\ltimes_d I$, we describe the structures of ideals of  $A\ltimes_d I$.
In \seref{sec3},  we investigate the subcoalgebras of Dorroh extensions of coalgebras.
Given a coalgebra Dorroh extension $C\ltimes_d P$, we describe the structures of
subcoalgebra of $C\ltimes_d P$. In \seref{sec4}, as an application, we compute the ideals of unitization $k\ltimes_d I$
of an algebra $I$ without identity and the ideals of the trivial extension $A\ltimes M$ of an algebra $A$.
We also describe all subcoalgebras of  counitization $k\ltimes_d P$ of a coalgebra $P$ without counit
and the subcoalgebras of the trivial extension $C\ltimes M$ of a coalgebra $C$.

Throughout this paper, we work over a field $k$. All algebras are not necessarily unital, and coalgebras are not necessarily counital as in the previous paper \cite{YC}, unless otherwise stated.
However, Hopf algebras have always identities and counits.
For basic facts about coalgebras and Hopf algebras, the reader can refer to the books \cite{Abe, Mo, Sw}.
For simplicity, we use $1$ denote the identity maps on vector spaces.
When an algebra $A$ is unital, a left (resp., right) $A$-module $M$ is always assumed to be unital, i.e., $1_Am=m$
(resp., $m1_A=m$) for any $m\in M$.
For a coalgebra $C$ and $c\in C$, we write $\D(c)=\sum c_1\ot c_2$.
If $(M, \rho)$ is a right (resp., left) $C$-comodule, we write $\rho(m)=\sum m_{(0)}\ot m_{(1)}$
(resp.,  $\rho(m)=\sum m_{(-1)}\ot m_{(0)}$), $m\in M$. When $C$ is counital, a right (resp., left) $C$-comodule $M$ is always assumed to be counital, i.e., $\sum m_{(0)}\ep(m_{(1)})=m$
(resp.,  $\sum\ep(m_{(-1)})m_{(0)}=m$), $m\in M$.

\section{Dorroh extensions of Hopf algebras}\selabel{sec1}

In this section, a bialgebra does not necessarily have an identity nor a counit.

Recall from \cite{YC} that $(H,I)$ is called a {\it Dorroh pair of algebras} if  $H$ and $I$ are algebras and $I$ is an $H$-bimodule such that the module actions
are compatible with the multiplication of $I$, that is,
\begin{equation*}
	a(xy)=(ax)y, \quad  (xa)y=x(ay), \quad   (xy)a=x(ya),
\end{equation*}
for all $a\in H$ and $x, y\in I$. In this case, one can construct an associative algebra $H\ltimes_d I$ as follows:
$H\ltimes_d I=H\op I$ as a vector space,
the multiplication of $H\ltimes_d I$ is given by
\begin{equation*}
(a, x)(b, y)=(ab, ay+xb+xy).
\end{equation*}
$H\ltimes_d I$ is an algebra Dorroh extension of $H$ by $I$, where $H$ and $I$ are regarded as subspaces of $H\ltimes_d I$ in a canonical way.
Moreover, any algebra Dorroh extension of $H$ is isomorphic to some  $H\ltimes_d I$ as an algebra.

Dually, $(H, I)$ is called a {\it Dorroh pair of coalgebras} if $H$ and $I$ are coalgebras and $I$ is an $H$-bicomodule such that the comodule coactions are compatible with the comultiplication of $I$, that is,
\begin{eqnarray*}
\sum x_1\ot x_{2(0)}\ot x_{2(1)}&=&\sum x_{(0)1}\ot x_{(0)2}\ot x_{(1)}, \\
\sum x_{1(-1)}\ot x_{1(0)}\ot x_2&=&\sum x_{(-1)}\ot x_{(0)1}\ot x_{(0)2}, \\
\sum x_{1(0)}\ot x_{1(1)}\ot x_{2}&=&\sum x_1\ot x_{2(-1)}\ot x_{2(0)},
\end{eqnarray*}
for all $x\in I$. In this case, one can a coassociative coalgebra $H\ltimes_d I$ as folows: $H\ltimes_d I=H\op I$ as vector spaces, the comultiplication is given by
\begin{equation*}
\begin{split}
\D(h,x)=&\sum (h_1,0)\ot (h_2,0)+\sum (x_{(-1)},0)\ot (0,x_{(0)})\\
&+\sum (0,x_{(0)})\ot (x_{(1)},0)+\sum (0,x_1)\ot (0,x_2).
\end{split}
\end{equation*}
$H\ltimes_d I$ is a coalgebra Dorroh extension of $H$ by $I$.
Moreover, any coalgebra Dorroh extension of $H$ is isomorphic to some $H\ltimes_d I$ as a coalgebra.

For a Dorroh pair $(H, I)$ of algebras or coalgebras, let $\tau_H: H\ra H\ltimes_d I$, $h\mapsto(h,0)$ and $\tau_I: I\ra H\ltimes_d I$, $x\mapsto (0, x)$
be the canonical injections, respectively, and let $\pi_H: H\ltimes_d I\ra H$, $(h, x)\mapsto h$ and $\pi_I: H\ltimes_d I\ra I$, $(h, x)\mapsto x$ be the canonical projections, respectively. For any subspaces $U\ss H$ and $V\ss I$, denote $\tau_H(U)$ by $(U,0)$, $\tau_I(V)$ by $(0,V)$ and $\tau_A(U)+\tau_I(V)$ by $(U,V)$. In what follows, we will always use such symbols.

\begin{definition}\delabel{bialg}
Let $H$ be a bialgebra. A bialgebra $B$ is called a bialgebra Dorroh extension of $H$ if $H$
is a subbialgebra of $B$ and there exists a biideal $I$ of $B$ such that $B=H\op I$ as vector spaces.	
In this case, $B$ is also called a bialgebra Dorroh extension of $H$ by $I$.
\end{definition}

Let $(H,I)$ be both a Dorroh pair of algebras and a Dorroh pair of coalgebras.
Then $H\ltimes_d I$ is both an algebra and a
coalgebra as stated above. We will give some necessary and sufficient conditions for $H\ltimes_d I$
to be a bialgebra or a Hopf algebra.

\begin{proposition}\prlabel{prop1.2}
Let $(H,I)$ be both a Dorroh pair of algebras and a Dorroh pair of coalgebras.
Then $H\ltimes_d I$ is a bialgebra if and only if the following are satisfied:
\begin{enumerate}
\item[(a)] $H$ is a bialgebra, i.e., $\forall a, b\in H$,
      \begin{equation}\eqlabel{e1}
        \sum (ab)_1\ot (ab)_2=\sum a_1b_1\ot a_2b_2.
      \end{equation}
\item[(b)] The comultiplication of $I$, the left and right $H$-comodule structure maps of $I$ are all $H$-bimodule homomorphism, i.e.,
$\forall a, b\in H$ and $x, y\in I$,
    \begin{eqnarray}
\eqlabel{e2} \sum (ay)_1\ot (ay)_2&=&\sum a_1 y_1\ot a_2 y_2,\\
\eqlabel{e3} \sum (xb)_1\ot (xb)_2&=&\sum x_1b_1\ot x_2b_2,\\
\eqlabel{e4} \sum (ay)_{(-1)}\ot (ay)_{(0)}&=&\sum a_1 y_{(-1)}\ot a_2 y_{(0)}, \\
\eqlabel{e5} \sum (xb)_{(-1)}\ot (xb)_{(0)}&=&\sum x_{(-1)}b_1\ot x_{(0)}b_2, \\
\eqlabel{e6} \sum (ay)_{(0)}\ot (ay)_{(1)}&=&\sum a_1 y_{(0)}\ot a_2 y_{(1)}, \\
\eqlabel{e7} \sum (xb)_{(0)}\ot (xb)_{(1)}&=&\sum x_{(0)}b_1\ot x_{(1)}b_2.
  \end{eqnarray}
\item[(c)]  The left and right $H$-comodule structure maps of $I$ are both algebra homomorphisms,
i.e., $\forall x, y\in I$,
 \begin{eqnarray}
   \eqlabel{e8} \sum (xy)_{(-1)}\ot (xy)_{(0)}&=&\sum x_{(-1)} y_{(-1)}\ot x_{(0)} y_{(0)}, \\
  \eqlabel{e9} \sum (xy)_{(0)}\ot (xy)_{(1)}&=&\sum x_{(0)} y_{(0)}\ot x_{(1)} y_{(1)}.
\end{eqnarray}
\item[(d)] The all structure maps of $I$ (multiplication, comultiplication, $H$-bimodule and $H$-bicocomodule)
satisfy the following compatible condition: $\forall x, y\in I$,
 \begin{equation}\eqlabel{e10}
 \begin{split}
\sum (xy)_1\ot (xy)_2=&\sum x_{(-1)} y_{(0)}\ot x_{(0)} y_{(1)}+\sum x_{(-1)} y_1\ot x_{(0)}y_2
+\sum x_{(0)} y_{(-1)}\ot x_{(1)} y_{(0)}\\
&+\sum x_{(0)} y_1\ot x_{(1)} y_2
+\sum x_1 y_{(-1)}\ot x_2 y_{(0)}+\sum x_1 y_{(0)}\ot x_2 y_{(1)}\\
&+\sum x_1y_1\ot x_2y_2.
\end{split}
\end{equation}
\end{enumerate}
In this case, if we regard $H$ and $I$ as subspaces of $H\ltimes_d I$ via $\tau_H$ and $\tau_I$ respectively,
then $H\ltimes_d I$ is a bialgebra Dorroh extension of $H$ by $I$.
\end{proposition}

\begin{proof}
Note that $H\ltimes_d I$ is a bialgebra if and only if $\D((a,x)(b,y))=\D(a,x)\D(b,y)$ for any
$(a,x), (b,y)\in H\ltimes_d I$. Clearly, $\D((a,x)(b,y))=\D(a,x)\D(b,y)$ if and only if
$\D((a,0)(b,0))=\D(a,0)\D(b,0)$,  $\D((a,0)(0,y))=\D(a,0)\D(0,y)$,
$\D((0,x)(b,0))=\D(0,x)\D(b,0)$ and $\D((0,x)(0,y))=\D(0,x)\D(0,y)$.
Then a straightforward computation shows that $\D((a,0)(b,0))=\D(a,0)\D(b,0)$ is equivalent to Eq.\equref{e1}. Similarly, $\D((a,0)(0,y))=\D(a,0)\D(0,y)$ is equivalent to the three equations Eq.\equref{e2}, Eq.\equref{e4} and Eq.\equref{e6},
$\D((0,x)(b,0))=\D(0,x)\D(b,0)$ is equivalent to the three equations Eq.\equref{e3}, Eq.\equref{e5} and Eq.\equref{e7},
and $\D((0,x)(0,y))=\D(0,x)\D(0,y)$ is equivalent to the three equations Eq.\equref{e8}, Eq.\equref{e9} and Eq.\equref{e10}.

Now assume that $H\ltimes_d I$ is a bialgebra. Then $H$ is also a bialgebra by (a).
By regarding $H=\tau_H(H)$ and $I=\tau_I(I)$, one can see that $H$ is a subbialgebra of $H\ltimes_d I$ and $I$ is a biideal of $H\ltimes_d I$.
Therefore, $H\ltimes_d I$ is a bialgebra Dorroh extension of $H$ by $I$.
\end{proof}

\begin{definition}
A pair $(H,I)$ as stated in \prref{prop1.2} is called a Dorroh pair of bialgebras.	
\end{definition}

\begin{remark}
Let $(H,I)$ be a Dorroh pair of bialgebras. Then
	
(a) Though $I$ is both an algebra and a coalgebra, $I$ is not a bialgebra
unless $\sum x_{(-1)} y_{(0)}\ot x_{(0)} y_{(1)}+\sum x_{(-1)} y_1\ot x_{(0)}y_2
 +\sum x_{(0)} y_{(-1)}\ot x_{(1)} y_{(0)}+\sum x_{(0)} y_1\ot x_{(1)} y_2
 +\sum x_1 y_{(-1)}\ot x_2 y_{(0)}+\sum x_1 y_{(0)}\ot x_2 y_{(1)}=0$ for all $x, y\in I$.

(b) $I$ is an $H$-Hopf bimodule  by Eqs.\equref{e4}-\equref{e7}. $I$ is a left $H$-module coalgebra
and a right $H$-module coalgebra by Eq.\equref{e2} and Eq.\equref{e3}. Moreover,
$I$ is a left $H$-comodule algebra and a right $H$-comodule algebra by Eq.\equref{e8} and Eq.\equref{e9}.

(c) The comultiplication $\D_I$ of $I$ is different from $\D_{H\ltimes_d I}|_I$,
the restriction of the comultiplication $\D_{H\ltimes_d I}$ of $H\ltimes_d I$ on $I$.
\end{remark}

\begin{proposition}\prlabel{prop1.5}
Let $A$ be a bialgebra Dorroh extension of $H$ by $I$. Then $(H,I)$ is a Dorroh pair of bialgebras
and $A\cong H\ltimes_d I$ as bialgebras, where the comultiplication $\D_I$ and the comodule structure maps $\rho_l$ and $\rho_r$
of $I$ are given respectively by the following compositions
$$\begin{array}{c}
\D_I: I\hookrightarrow A\xrightarrow{\D_A}A\ot A\xrightarrow{\pi_I\ot \pi_I}I\ot I ,\\
\rho_l: I\hookrightarrow A\xrightarrow{\D_A}A\ot A\xrightarrow{\pi_H\ot \pi_I}H\ot I ,\\
\rho_r: I\hookrightarrow A\xrightarrow{\D_A}A\ot A\xrightarrow{\pi_I\ot \pi_H}I\ot H ,\\
\end{array}$$
$\pi_H: A\ra H$ and $\pi_I: A\ra I$ are the projections
corresponding to the direct sum decomposition $A=H\op I$ of vector spaces.
\end{proposition}
\begin{proof}
By \cite{YC}, one knows that $(H,I)$ is both a Dorroh  pair of algebras and a Dorroh  pair of coalgebras.
Moreover, the canonical linear isomorphism $H\ltimes_dI\ra A$, $(h, x)\mapsto h+x$
is both an algebra isomorphism and a coalgebra isomorphism. This implies that $H\ltimes_dI$ is a bialgebra since $A$ is a bialgebra. Hence $(H,I)$ is a Dorroh  pair of bialgebras by \prref{prop1.2}.
\end{proof}

In the following, a Hopf algebra means a usual Hopf algebra which is unital as an algebra and counital as a coalgebra and
has an antipode. For a Hopf algebra $H$, let $1_H$ (or 1), $\ep_H$ (or $\ep$) and $S_H$ (or $S$) denote its identity, counit and antipode as usual,
respectively.

Note that $A$ is a bialgebra Dorroh extension of $H$ by $I$ if and only if there is a bialgebra projection $\pi$ from $A$ onto $H$ with ${\rm Ker}(\pi)=I$.
Radford considered a bialgebra with a projection onto a Hopf algebra in \cite{Rad}.

Let $(H,I)$ be a Dorroh pair of bialgebras and $A=H\ltimes_d I$. Then by \prref{prop1.2}, $A$ is a bialgebra Dorroh extension of $H$ by $I$.
Assume that $H$ is a Hopf algebra with antipode $S$.
Then $A$ is a usual bialgebra with identity $(1_H,0)$ and counit $(\ep_H, 0)$, $\pi_H: A\ra H$ and $\tau_H: H\ra A$ are bialgebra maps and $\pi_H\circ\tau_H=1$. Define $\Pi={\rm id}_A*(\tau_H\circ S\circ\pi_H)$ in the convolution algebra ${\rm Hom}(A,A)$ and $B=\Pi(A)$. Then $B$ is a subalgebra of $A$ and $B$ has a (unique) coalgebra structure such that $\Pi: A\ra B$ is a coalgebra map. Moreover, $B=\{a\in A\mid (1\ot\pi_H)\D(a)=a\ot 1\}$ is a braided bialgebra in the category $_H^H \mathcal{YD}$ of Yetter-Drinfeld $H$-modules and $A\cong B\sharp H$, the biproduct of $B$ by $H$, as bialgebras, see \cite{Rad}.

In the following proposition, we describe the relation between the braided bialgebra $B$ and the biideal $I$.

\begin{proposition}\prlabel{}
Assume that $H$ a Hopf algebra with antipode $S$.
Then the following hold.
\begin{enumerate}
\item[(a)] Let $(H,I)$ be a Dorroh pair of bialgebras and $A=H\ltimes_dI$. Then with the above notations, $B=(k1,I^{coH})$ and $A\cong B\sharp H$ as bialgebras, where $I^{coH}=\{x\in I\mid\rho_r(x)=x\ot 1\}$, the $H$-coinvariants of $I$.
\item[(b)] Let $B$ be a braided bialgebra in the Yetter-Drinfeld category $_H^H \mathcal{YD}$ and $A=B\sharp H$. Then $(H, I)$ is a Dorroh pair of bialgebras and $A\cong H\ltimes_dI$ as bialgebras, where $I=\Ker(\ep_B)\sharp H$ and $\ep_B$ is the counit of $B$.
\end{enumerate}
\end{proposition}

\begin{proof}
(a) Let $(h,x)\in A$. Then $(h,x)\in B$ if and only if $(1\ot\pi_H)\D(h,x)=(h,x)\ot 1$. Now we have $(1\ot\pi_H)\D(h,x)=\sum(h_1,0)\ot h_2+\sum(0,x_{(0)})\ot x_{(1)}$.
Suppose $\sum(h_1,0)\ot h_2+\sum(0,x_{(0)})\ot x_{(1)}=(h,x)\ot 1$. Then by applying $\pi_H\ot 1$ and $\pi_I\ot 1$ to the equation respectively, one gets $\sum h_1\ot h_2=h\ot 1$ and $\sum x_{(0)}\ot x_{(1)}=x\ot 1$. Hence $h=\ep_H(h)1\in k1$ and $x\in I^{coH}$, i.e., $(h,x)\in(k1, I^{coH})$. Conversely, if $(h,x)\in(k1, I^{coH})$, then $(h,x)\in B$ by the discussion above. Therefore, $B=(k1,I^{coH})$
and $A\cong B\sharp H$ as bialgebras.

(b) Let $\iota: H\ra B\sharp H=A$ and $\pi: A=B\sharp H\ra H$ be defined by
$\iota(h)=1\sharp h$ and $\pi(b\sharp h)=\ep_B(b)h$ respectively, where $h\in H$ and $b\in B$. Then by \cite{Rad}, $\iota$ and $\pi$ are bialgebra maps and $\pi\circ\iota=1$. Regarding $\iota$ as an embedding and $H=\iota(H)$, then $A$ is a bialgebra Dorroh extension of $H$ by $I$ with $I={\rm Ker}(\pi)=\Ker(\ep_B)\sharp H$. By \prref{prop1.5}, $(H, I)$ is a Dorroh pair of bialgebras and $A\cong H\ltimes_dI$ as bialgebras.
\end{proof}

\begin{definition}\delabel{Hopfalg}
	Let $H$ be a Hopf algebra. A Hopf algebra $A$ is called a Hopf algebra Dorroh extension of $H$ if $H$
	is a Hopf subalgebra of $A$ and there exists a Hopf ideal $I$ of $A$ such that $A=H\op I$ as vector spaces.	
	In this case, $A$ is also called a Hopf algebra Dorroh extension of $H$ by $I$.
\end{definition}

Let $(H,I)$ be a Dorroh pair of bialgebras. Then by \prref{prop1.2},
$H\ltimes_d I$ is a bialgebra Dorroh extension of $H$ by $I$. That is, under the identifications
$H=\tau_H(H)$ and $I=\tau_I(I)$, $H$ is a subbialgebra of $H\ltimes_d I$
and $I$ is a biideal of $H\ltimes_d I$.

\begin{proposition}\prlabel{prop1.7}
Let $(H,I)$ be a Dorroh pair of bialgebras. Assume that $H$ is a Hopf algebra with antipode $S_H$.
Then $H\ltimes_d I$ is a Hopf algebra Dorroh extension of $H$ by $I$ if and only if
there exists a linear endomorphism $S_I$ of $I$ such that for any $x\in I$,
$$\begin{array}{l}
\sum S_H(x_{(-1)})x_{(0)}+\sum S_I(x_{(0)})x_{(1)}+\sum S_I(x_1)x_2=0,\\
\sum x_{(-1)}S_I(x_{(0)})+\sum x_{(0)}S_H(x_{(1)})+\sum x_1S_I(x_2)=0.\\
\end{array}$$
\end{proposition}
\begin{proof}
By \prref{prop1.2}, $H\ltimes_d I$ is a bialgebra Dorroh extension of $H$ by $I$,
where $H$ and $I$ are regarded as subspaces of $H\ltimes_d I$ under the identifications $H=\tau_H(H)$ and $I=\tau_I(I)$.
Since $H$ is a Hopf algebra, $H\ltimes_d I$ has an identity $(1_H, 0)$ and a counit
$(\ep_H, 0)$.

Assume that there exists a linear endomorphism $S_I$ of $I$ such that
$$\begin{array}{c}
\sum S_H(x_{(-1)})x_{(0)}+\sum S_I(x_{(0)})x_{(1)}+\sum S_I(x_1)x_2=0,\\
\sum x_{(-1)}S_I(x_{(0)})+\sum x_{(0)}S_H(x_{(1)})+\sum x_1S_I(x_2)=0\\
\end{array}$$
for any $x\in I$. Define  a linear map $S: H\ltimes_d I\ra H\ltimes_d I$ by
$S(h, x)=(S_H(h), S_I(x))$ for any $(h,x)\in H\ltimes_d I$. Then for any $(h,x)\in H\ltimes_d I$,
$$\begin{array}{rl}
\sum S((h,x)_1)(h,x)_2=&\sum S(h_1, 0)(h_2, 0)+\sum S(x_{(-1)}, 0)(0, x_{(0)})\\
&+\sum S(0, x_{(0)})(x_{(1)}, 0)+\sum S(0, x_1)(0, x_2)\\
=&\sum (S_H(h_1), 0)(h_2, 0)+\sum (S_H(x_{(-1)}), 0)(0, x_{(0)})\\
&+\sum(0, S_I(x_{(0)}))(x_{(1)}, 0)+\sum (0, S_I(x_1))(0, x_2)\\
=&\sum (S_H(h_1)h_2, 0)+\sum (0, S_H(x_{(-1)})x_{(0)})\\
&+\sum(0, S_I(x_{(0)})x_{(1)})+\sum (0, S_I(x_1)x_2)\\
=&(\sum S_H(h_1)h_2, \sum S_H(x_{(-1)})x_{(0)}
+\sum S_I(x_{(0)})x_{(1)}+\sum S_I(x_1)x_2)\\
=&(\ep_H(h)1_H, 0)=\ep(h,x) (1_H, 0),
\end{array}$$
and similarly $\sum(h, x)_1S((h,x)_2)=\ep(h,x) (1_H, 0)$.
Thus, $H\ltimes_d I$ is a Hopf algebra. Clearly, $H$ is a Hopf subalgebra of $H\ltimes_d I$
and $I$ is a Hopf ideal of $H\ltimes_d I$. Hence $H\ltimes_d I$ is a Hopf algebra Dorroh extension
of $H$ by $I$.

Conversely, assume that $H\ltimes_d I$ is a Hopf algebra Dorroh extension
of $H$ by $I$. Then $H$ is a Hopf subalgebra of $H\ltimes_d I$
and $I$ is a Hopf ideal of $H\ltimes_d I$. Hence $S(h, 0)=(S_H(h), 0)$
for any $h\in H$, and there is a linear endomorphism $S_I$ of $I$ such that
$S(0,x)=(0, S_I(x))$ for any $x\in I$, where $S$ is the antipode of $H\ltimes_d I$.
Now let $x\in I$. Then $\sum S((0,x)_1)(0,x)_2=\ep(0,x) (1_H, 0)=0$.
On the other hand, we have
$$\begin{array}{rl}
\sum S((0,x)_1)(0,x)_2
=&\sum S(x_{(-1)}, 0)(0, x_{(0)})+\sum S(0,x_{(0)})(x_{(1)},0)+\sum S(0, x_1)(0, x_2)\\
=&\sum (S_H(x_{(-1)}), 0)(0, x_{(0)})+\sum (0,S_I(x_{(0)}))(x_{(1)},0)+\sum (0, S_I(x_1))(0, x_2)\\
=&\sum (0, S_H(x_{(-1)})x_{(0)})+\sum (0,S_I(x_{(0)})x_{(1)})+\sum (0, S_I(x_1)x_2)\\
=&(0, \sum S_H(x_{(-1)})x_{(0)}+\sum S_I(x_{(0)})x_{(1)}+\sum S_I(x_1)x_2).\\
\end{array}$$
It follows that $\sum S_H(x_{(-1)})x_{(0)}+\sum S_I(x_{(0)})x_{(1)}+\sum S_I(x_1)x_2=0$.
Similarly, from $\sum(0,x)_1S((0,x)_2)=\ep(0,x) (1_H, 0)=0$, one gets
$\sum x_{(-1)}S_I(x_{(0)})+\sum x_{(0)}S_H(x_{(1)})+\sum x_1S_I(x_2)=0$.
\end{proof}

\begin{definition}
A pair $(H,I)$ as stated in \prref{prop1.7} is called a Dorroh pair of Hopf algebras.
\end{definition}

\begin{corollary}\colabel{coro1.9}
Let $(H,I)$ be a Dorroh pair of Hopf algebras. With the notations of  \prref{prop1.7},
for any $h\in H$ and $x, y\in I$, we have
\begin{enumerate}
\item[(a)] $S_I(hx)=S_I(x)S_H(h)$, $S_I(xh)=S_H(h)S_I(x)$ and $S_I(xy)=S_I(y)S_I(x)$,
\item[(b)] $\sum S_I(x)_{(0)}\ot S_I(x)_{(1)}=\sum S_I(x_{(0)})\ot S_H(x_{(-1)})$,
$\sum S_I(x)_{(-1)}\ot S_I(x)_{(0)}=\sum S_H(x_{(1)})\ot S_I(x_{(0)})$ and
$\sum S_I(x)_1\ot S_I(x)_2=\sum S_I(x_2)\ot S_I(x_1)$.
\end{enumerate}
\end{corollary}
\begin{proof}
By \prref{prop1.7}, $H\ltimes_d I$ is a Hopf algebra with the antipode $S$ given by
$S(h, x)=(S_H(h), S_I(x))$, $(h, x)\in H\ltimes_d I$. Let $h\in H$ and $x, y\in I$.
Since $S$ is an algebra antimorphism, we have $S((h,0)(0,x))=S(0,x) S(h,0)$,
$S((0,x)(h,0))=S(h,0)S(0,x)$ and $S((0,x)(0,y))=S(0,y) S(0,x)$. They are equivalent to
$S_I(hx)=S_I(x)S_H(h)$, $S_I(xh)=S_H(h)S_I(x)$ and $S_I(xy)=S_I(y)S_I(x)$, respectively.
This shows (a).
Since $S$ is also a coalgebra antimorphism, we have
$\sum (S(0,x))_1\ot (S(0,x))_2=\sum S((0,x)_2)\ot S((0,x)_1)$, that is,
$$\begin{array}{rl}
&\sum (S_I(x)_{(-1)}, 0)\ot(0, S_I(x)_{(0)})+\sum(0, S_I(x)_{(0)})\ot(S_I(x)_{(1)}, 0)
+\sum(0, S_I(x)_1)\ot (0, S_I(x)_2)\\
=&\sum(0, S_I(x_{(0)}))\ot(S_H(x_{(-1)}), 0)+\sum(S_H(x_{(1)}), 0)\ot(0, S_I(x_{(0)}))
+\sum(0, S_I(x_2))\ot(0, S_I(x_1)).\\
\end{array}$$
Applying $\pi_I\ot\pi_H$, $\pi_H\ot\pi_I$ and $\pi_I\ot\pi_I$ to the above equation respectively,
one gets $\sum S_I(x)_{(0)}\ot S_I(x)_{(1)}=\sum S_I(x_{(0)})\ot S_H(x_{(-1)})$,
$\sum S_I(x)_{(-1)}\ot S_I(x)_{(0)}=\sum S_H(x_{(1)})\ot S_I(x_{(0)})$ and
$\sum S_I(x)_1\ot S_I(x)_2=\sum S_I(x_2)\ot S_I(x_1)$. This shows (b).
\end{proof}

\begin{proposition}\prlabel{prop1.10}
Let $A$ be a Hopf algebra Dorroh extension of $H$ by $I$. Then $(H,I)$ is a Dorroh pair of Hopf algebras
and $A\cong H\ltimes_d I$ as Hopf algebras, where the comultiplication $\D_I$ of $I$ and the
comodule structure maps of $I$ are given in \prref{prop1.5}.
\end{proposition}
\begin{proof}
By \prref{prop1.5} and its proof, $(H,I)$ is a Dorroh pair of bialgebras
and there is a bialgebra isomorphism $\phi: H\ltimes_d I\ra A$, $(h,x)\mapsto h+x$.
Since $A$ is a Hopf algebra, $H\ltimes_d I$ is also a Hopf algebra with the antipode $S$
given by $S=\phi^{-1}\circ S_A\circ\phi$ and $\phi$ is a Hopf algebra isomorphism, where $S_A$ is the antipode of $A$.
Since $A$ is a Hopf algebra Dorroh extension of $H$ by $I$,
$H\ltimes_d I$ is a Hopf algebra Dorroh extension of $\phi^{-1}(H)$ by $\phi^{-1}(I)$.
Clearly, $\phi^{-1}(H)=\tau_H(H)$ and $\phi^{-1}(I)=\tau_I(I)$.
Thus, under the identifications $H=\tau_H(H)$ and $I=\tau_I(I)$, $H\ltimes_dI$ is a Hopf algebra Dorroh extension of $H$ by $I$, and so $(H,I)$ is a Dorroh pair of Hopf algebras
by \prref{prop1.7}.
\end{proof}

In the rest of this section, we give some examples of Dorroh pairs (or Dorroh extensions) of
bialgebras and Hopf algebras.

\begin{example}
(a) Let $I$ be both an algebra and a coalgebra. Assume that the multiplication and the comultiplication of $I$ satisfy
$$\begin{array}{c}
\D_I(xy)=x\ot y+y\ot x+\sum y_1\ot xy_2+\sum xy_1\ot y_2
+\sum x_1\ot x_2y+\sum x_1y\ot x_2+\sum x_1y_1\ot x_2y_2\\
\end{array}$$
for all $x, y\in I$.
Here $I$ is not necessarily a bialgebra. Note that the ground field $k$ is a Hopf algebra as usual.
Hence $I$ is a $k$-Hopf bimodule with the trivial comodule structures.
One can check that $(k,I)$ is a Dorroh pair of bialgebras, and so
$k\ltimes_d I$ is a usual bialgebra with the identity $(1,0)$ and the counit $\ep$ given by $\ep (\a,x)=\a$,  $(\a,x)\in k\ltimes_d I$.
Moreover, the multiplication and comultiplication of $k\ltimes_d I$ are given by
\begin{align*}
&(\a,x)(\beta,y)=(\a\beta, \a y+\beta x+xy),\\
&\D(\a, x)=(\a,0)\ot (1,0)+(1,0)\ot (0,x)+(0,x)\ot (1,0)+\sum(0,x_1)\ot (0,x_2).
\end{align*}
Clearly $(\a, x)$ is a group-like element if and only if $\a=1$ and $\D_I(x)=x\ot x$.

If there exists a linear map $S_I: I\ra I$ such that
$\sum S_I(x_1)x_2=\sum x_1S_I(x_2)=-S_I(x)-x$,
then it follows from \prref{prop1.7} that $k\ltimes_d I$ is a Hopf algebra
with the antipode $S$ given by $S(\a,x)=(\a, S_I(x))$.

(b)
Let $H$ be a bialgebra and $M$ an $H$-Hopf bimodule. Define a multiplication on $M$ by $xy=0$, $x, y\in M$,
and a comultiplication on $M$ by $\D_M=0$ , respectively. Then $M$ is both an algebra
and a coalgebra. From \prref{prop1.2} and \prref{prop1.7}, one obtains the results
in \cite[Theorem 11 and Corollary 12]{ZT} as follows.
\begin{enumerate}
\item [(i)] The trivial extension $H\ltimes M$ is a bialgebra if and only if $M$ is an $H$-Hopf bimodule
and satisfies $\sum m_{(-1)}x_{(0)}\ot m_{(0)}x_{(1)} +\sum m_{(0)}x_{(-1)}\ot m_{(1)}x_{(0)}=0$ for all $m, x\in M$.
\item [(ii)] Suppose the trivial extension $H\op M$ is
a bialgebra. If $H$ is a Hopf algebra, then so is $H\ltimes M$.
\end{enumerate}

(c) Let $H=\op_{i=0}^{\infty} H_i$ be an $\mathbb{N}$-graded Hopf algebra. Set $I:=\op_{i=1}^{\infty} H_i$. Then $(H_0,I)$ is a Dorroh pair of Hopf algebras and $H\cong H_0\ltimes_d I$ as Hopf algebras.
\end{example}

\section{The ideals of algebra Dorroh extensions}\selabel{sec2}

Throughout this section, let $(A, I)$ be a Dorroh pair of algebras.
Then $A\ltimes_d I$ is an algebra Dorroh extension of $A$ by $I$ as stated in the last section. we will study the ideals of $A\ltimes_d I$.

Let $K\ss A\ltimes_d I$ be a subspace.
Put
\begin{eqnarray*}
B&:=&\{a\in A|(a,x)\in K \text{ for some } x\in I\},\\
J&:=&\{x\in I|(a,x)\in K \text{ for some } a\in A\},\\
Z&:=&\{a\in A|(a,0)\in K\},\\
L&:=&\{x\in I|(0,x)\in K\}.
\end{eqnarray*}
Then $B$ and $Z$ are subspaces of $A$ and $Z\ss B$, $J$ and $L$ are subspaces of $I$ and $L\ss J$.
Define a linear map  $\varphi: J\ra B/Z\ss A/Z$ by $\varphi(x)=a+Z$ if $(a,-x)\in K$.
Obviously, $\varphi$ is a well-defined linear map, and $\mathrm{Im}(\varphi)=B/Z$ and $\mathrm{Ker}(\varphi)=L$.
Thus, there is an induced isomorphism $\overline{\varphi}:J/L\ra B/Z$ such that
$\overline{\varphi}(x+L)=a+Z$ when $(a,-x)\in K$.
In this case, the subspace $K$ can be described as follows:
$$\begin{array}{rl}
K&=\{(a,-x)\in A\ltimes_d I|a\in B, x\in J , \varphi(x)=a+Z\}\\
&=\{(a,-x)\in A\ltimes_d I|a\in B, x\in J , \overline{\varphi}(x+L)=a+Z\}.\\
\end{array}$$

Similarly to the description of ideals of Dorroh extensions of rings in
\cite[Proposition 5]{M}, one can describe the ideals of Dorroh extensions of algebras in the following.

\begin{proposition}\prlabel{prop2.1}
With the above notations, $K$ is an ideal of $A\ltimes_d I$ if and only if
the following (a), (b) and (c) hold if and only if the following (a), (b') and (c') hold.
\begin{enumerate}
\item[(a)]  $Z$ and $B$ are both ideals of $A$, so $B/Z$ is a well-defined $A$-Algebra.
\item[(b)] $J$ is an $A$-subAlgebra of $I$.
\item[(b')] $J$ is an $A$-subbimodule of $I$.
\item[(c)] $\varphi: J\ra B/Z$ is an $A$-Algebra homomorphism,
and if $\varphi(x)=a+Z$ then $ay-xy, ya-yx\in \mathrm{Ker}(\varphi)$ for any $y\in I$.
\item[(c')] $\varphi: J\ra B/Z$ is an $A$-bimodule homomorphism,
and if $\varphi(x)=a+Z$ then $ay-xy, ya-yx\in \mathrm{Ker}(\varphi)$ for any $y\in I$.
\end{enumerate}
\end{proposition}

\begin{proof}
It is similar to the proof of \cite[Proposition 5]{M}. Here we only prove that
(a), (b') and (c') imply that $K$ is an ideal of $A\ltimes_d I$.

Let $(a,-x)\in K$. Then $a\in B$, $x\in J$ and $\varphi(x)=a+Z$.
Let $b\in A$ and $y\in I$. Then $ab\in B$ by (a), and $xb\in J$ by (b').
By (c'), $\varphi(xb)=\varphi(x)b=(a+Z)b=ab+Z$, and hence $(ab,-xb)\in K$.
Again by (c'), $ay-xy\in{\rm Ker}(\varphi)$. Hence $ay-xy\in J$
and $\varphi(ay-xy)=0+Z$. This implies $(0, -ay+xy)\in K$. Therefore,
$(a,-x)(b,y)=(ab, ay-xb-xy)=(ab,-xb)-(0, -ay+xy)\in K$. Similarly,
one can show that $(b,y)(a,-x)\in K$. Thus, $K$ is an ideal of $A\ltimes_d I$.
\end{proof}

By \prref{prop2.1}, for any subspaces $Z\ss B\ss  A$ and $J\ss I$, if there exists a linear surjection $\varphi: J\ra B/Z$
such that (a), (b') and (c') (or (a), (b) and (c)) in \prref{prop2.1} are satisfied, then the following set is an ideal of $A\ltimes_d I$:
$$\{(a,-x)\in A\ltimes_d I|a\in B, x\in J , \varphi(x)=a+Z\}.$$
Conversely, any ideal of $A\ltimes_d I$ has this form.

We can also use the induced isomorphism $\overline{\varphi}:J/L\ra B/Z$ to replace the homomorphism $\varphi$.

\begin{corollary}\colabel{coro2.2}
With the above notations, $K$ is an ideal of $A\ltimes_d I$ if and only if
the following (a), (b) and (c) hold if and only if the following (a), (b') and (c') hold.
\begin{enumerate}
\item[(a)]  $Z$ and $B$ are both ideals of $A$, so $B/Z$ is a well-defined $A$-Algebra.
\item[(b)] $J$ is an $A$-subAlgebra of $I$ and $L$ is an $A$-Ideal of $I$, so $J/L$ is a well-defined $A$-Algebra.
\item[(b')] $J$ and $L$ are both $A$-subbimodules of $I$, so $J/L$ is a well-defined $A$-bimodule.
\item[(c)] $\overline{\varphi}: J/L\ra B/Z$ is an $A$-Algebra isomorphism,
and if $\overline{\varphi}(x+L)=a+Z$ then $ay-xy, ya-yx\in L$ for any $y\in I$.
\item[(c')] $\overline{\varphi}: J/L\ra B/Z$ is an $A$-bimodule isomorphism,
and if $\overline{\varphi}(x+L)=a+Z$ then $ay-xy, ya-yx\in L$ for any $y\in I$.
\end{enumerate}
\end{corollary}

\begin{proof}
It follows from \prref{prop2.1} and $\Ker (\varphi)=L$.
\end{proof}

Let $\tau_A: A\ra A\ltimes_d I$ and
$\tau_I: I\ra A\ltimes_dI$ be the canonical inclusions,
and let $\pi_A: A\ltimes_dI\ra A$ and
$\pi_I: A\ltimes_d I\ra I$ be the canonical projections as in the last section.

\begin{lemma}\lelabel{lem2.3}
With the notations above, $\tau_A$, $\tau_I$ and $\pi_A$ are algebra homomorphisms,
and the sequences
$$0\ra L\overset{\tau_I}{\ra} K\overset{\pi_A}{\ra} B\ra 0,\quad
0\ra Z\overset{\tau_A}{\ra} K\overset{\pi_I}{\ra} J\ra 0$$
are exact, where $K$ is a subspace of $A\ltimes_d I$, and $Z\ss B\ss A$ and $L\ss J\ss I$
are subspaces constructed from $K$ as before.
\end{lemma}
\begin{proof}
By the definition of $A\ltimes_d I$, one can see that $\tau_A$, $\tau_I$ and $\pi_A$ are algebra homomorphisms.
By the definitions of $B$, $J$, $Z$ and $L$, we have $\pi_A(K)=B$, $(0,I)\cap K=(0,L)$,
$\pi_I(K)=J$ and $(A,0)\cap K=(Z,0)$. Hence the exactness of the sequences in the lemma follows
form that of the sequences
$$0\ra I\overset{\tau_I}{\ra} A\ltimes_d I\overset{\pi_A}{\ra} A\ra 0,\quad
0\ra A\overset{\tau_A}{\ra} A\ltimes_d I\overset{\pi_I}{\ra} I\ra 0.$$
\end{proof}

By \leref{lem2.3}, we have  an algebra isomorphism $K/(0,L)\cong B$
and a vector space isomorphism $K/(Z,0)\cong J$.
Note that $\pi_I$ is generally not an algebra homomorphism unless $AI=IA=0$, i.e., the module actions are zero.

\begin{proposition}\prlabel{prop2.4}
Assume that $K$ is an ideal of $A\ltimes_d I$. Let $Z\subseteq B\subseteq A$ and $L\subseteq J\subseteq I$
be subspaces constructed from $K$ as before.
Then $(Z,L)$ is an ideal of $K$ and
$K/(Z,L)\cong B/Z \cong J/L$
as algebras.
\end{proposition}
\begin{proof}
Let $z\in Z$, $x\in L$ and $(a,-y)\in K$. Then by \coref{coro2.2}(a) and (c),
we have $az, za\in Z$ and  $ax-yx, xa-xy\in L$.
Since $(z,0)\in K$, $(0,y)(z,0)=(0,yz)\in K$ and $(z,0)(0,y)=(0,zy)\in K$,
and so $zy, yz\in L$.
Hence $(a,-y)(z,x)=(az,ax-yz-yx)\in(Z, L)$ and $(z,x)(a,-y)=(za, xa-zy-xy)\in(Z, L)$.
This shows that $(Z,L)$ is an ideal of $K$.

By \leref{lem2.3},  $\pi_A: K\ra B$ is an algebra epimorphism with ${\rm Ker}(\pi_A: K\ra B)=(0,L)$.
Since  $(Z,L)\supseteq(0,L)$ and $\pi_A(Z,L)=Z$,
$\pi_A$ induces an algebra isomorphism
$$\overline{\pi_A}: K/(Z,L)\ra B/Z,
(a,-y)+(Z,L)\mapsto a+Z.$$
By \coref{coro2.2}(c), $B/Z \cong J/L$ as algebras.
\end{proof}

\section{The subcoalgebras of coalgebra Dorroh extensions}\selabel{sec3}

Throughout this section, let $(C, P)$ be a Dorroh pair of coalgebras. Then $C\ltimes_d P$ is a coalgebra Dorroh extension of $C$ by $P$ as stated in \seref{sec1}. we shall study the subcoalgebras of $C\ltimes_d P$.

Let $T$ be a subspace of $C\ltimes_d P$. Put
\begin{eqnarray*}
D&:=&\{c\in C | (c,p)\in T \text{ for some } p\in P\},\\
Q&:=&\{p\in P | (c,p)\in T \text{ for some } c\in C\},\\
E&:=&\{c\in C | (c,0)\in T\},\\
R&:=&\{p\in P | (0,p)\in T\}.
\end{eqnarray*}
Then $D$ and $E$ are subspaces of $C$ and $E\ss D$, and $Q$ and $R$ are subspaces of $P$ and $R\ss Q$.
Define a linear map  $\eta: D\ra P/R $ by $\eta(c)=p+R$ if $(c,p)\in T$.
Clearly, $\eta$ is a well-defined linear map, $\mathrm{Im}(\eta)=Q/R$ and
$\mathrm{Ker}(\eta)=E$. Thus, $\eta$ induces a linear
isomorphism $\overline{\eta}: D/E\ra Q/R$ given by $\overline{\eta}(c+E)=p+R$ when $(c,p)\in T$.
In  this case, the subspace $T$ can be described as follows:
$$\begin{array}{rl}
T=&\{(c,p)\in C\ltimes_d P|c\in D, p\in Q, \eta(c)=p+R\}\\
=&\{(c,p)\in C\ltimes_d P|c\in D, p\in Q, \overline{\eta}(c+E)=p+R\}.\\
\end{array}$$

Let $\tau_C: C\ra C\ltimes_d P$ and
$\tau_P: P\ra C\ltimes_d P$ be the canonical inclusions,
and let $\pi_C: C\ltimes_d P\ra C$ and
$\pi_P: C\ltimes_d P\ra P$ be the canonical projections as before.

In the rest of this section, unless otherwise stated, we fix the above notations.

\begin{lemma}\lelabel{lem3.1}
$\tau_C$, $\pi_C$ and $\pi_P$ are coalgebra homomorphisms, and the sequences
$$0\ra R\overset{\tau_P}{\ra} T\overset{\pi_C}{\ra} D\ra 0,\quad
0\ra E\overset{\tau_C}{\ra} T\overset{\pi_P}{\ra} Q\ra 0$$
are exact.
\end{lemma}

\begin{proof}
It is straightforward to check that $\tau_C$, $\pi_C$ and $\pi_P$ are coalgebra homomorphisms.
The exactness of the two sequences follows from an argument similar to the proof of \leref{lem2.3}.
\end{proof}

Hence $T/(0,R)\cong D$, $T/(E,0)\cong Q$ as vector spaces.
Note that $\tau_P$ is generally not a coalgebra homomorphism.

Since $\pi_C$ is a coalgebra homomorphism, $C\ltimes_d P$ becomes a $C$-bicomodule
with the comodule structure maps given by $\rho_l=(\pi_C\ot 1)\D$ and $\rho_r=(1\ot\pi_C)\D$, respectively. That is,
$\rho_l(c,p)=(\pi_C\ot 1)\D (c,p)=\sum c_1\ot (c_2,0)+\sum p_{(-1)}\ot (0, p_{(0)})$
and $\rho_r(c,p)=( 1\ot\pi_C)\D (c,p)=\sum (c_1,0)\ot c_2+\sum (0,p_{(0)})\ot p_{(1)}$, where $(c, p)\in C\ltimes_dP$.
In this case, $\pi_P: C\ltimes_dP\ra P$ is a $C$-bicomodule homomorphism.
Obviously, $C\ltimes_d P$ is a $C$-Coalgebra.

Summarizing the discussion above, one gets the following lemma.

\begin{lemma}\lelabel{lem3.2}
$C\ltimes_d P$ is a $C$-Coalgebra with the comodule structure maps $\rho_l=(\pi_C\ot 1)\D$ and $\rho_r=(1\ot \pi_C)\D$. Moreover, $\pi_P: C\ltimes_d P\ra P$ is a $C$-Coalgebra homomorphism.
\end{lemma}

Suppose that $R$ is a $C$-Coideal of $P$. Then $P/R$ becomes a $C$-Coalgebra and $(C, P/R)$ is also a Dorroh pair of coalgebras. Moreover, the canonical
epimorphism $\pi: P\ra P/R$ is a $C$-Coalgebra homomorphism, which induces a
coalgebra epimorphism $(1,\pi): C\ltimes_dP\ra C\ltimes_d(P/R)$,
$(c,p)\mapsto(c,\pi(p))$. Clearly, ${\rm Ker}(1,\pi)=(0,R)$, a coideal
of $C\ltimes_dP$. Hence $(1,\pi)$ induces a coalgebra isomorphism
$(C\ltimes_dP)/(0,R)\ra C\ltimes_d(P/R)$,
$(c,p)+(0,R)\mapsto(c,\pi(p))=(c, p+R)$. We identify $(C\ltimes_dP)/(0,R)$ with $C\ltimes_d(P/R)$ via the isomorphism. In this case, the map
$(1,\pi): C\ltimes_dP\ra C\ltimes_d(P/R)$ is exactly the canonical epimorphism
$C\ltimes_dP\ra(C\ltimes_dP)/(0,R)$, denoted simply by $\pi$.
Moreover, for any $(c,p)\in T$, $\pi(c, p)=(c, p+R)=(c, \eta(c))$.
Let $\pi_C: C\ltimes_d(P/R)\ra C$ and $\pi_{P/R}: C\ltimes_d(P/R)\ra P/R$
denote the corresponding projections, respectively.

\begin{proposition}\prlabel{prop3.3}
$T$ is a subcoalgebra of $C\ltimes_dP$ if and only if the following hold:
\begin{enumerate}
\item[(a)] $D$ is a subcoalgebra of $C$.
\item[(b)] $Q$ is a $C$-subCoalgebra of $P$ and $R$ a $C$-Coideal of $Q$, and $\rho^P_l(Q)\subseteq D\ot Q$
and $\rho^P_r(Q)\subseteq Q\ot D$, where $\rho^P_l$ and $\rho^P_r$ are the $C$-comodule structure maps of $P$.
Hence $Q$ and $Q/R$ are both $D$-Coalgebras.
\item[(c)] $\eta: D\ra Q/R$ is a $C$-bicomodule (or $D$-bicomodule) homomorphism, and
$$\begin{array}{c}
\sum\eta(p_{(-1)})\ot p_{(0)} =\sum\pi(p_1)\ot p_2,\
\sum p_{(0)}\ot\eta(p_{(1)})=\sum p_1\ot\pi(p_2),\ \forall p\in Q.
\end{array}$$
\end{enumerate}
If this is the case, then $\eta$ is a coalgebra homomorphism, and hence it is
a $C$-Coalgebra (or $D$-Coalgebra) homomorphism.
\end{proposition}
\begin{proof}
Assume that $T$ is a subcoalgebra of $C\ltimes_dP$.
By \leref{lem3.1}, $D=\pi_C(T)$, $Q=\pi_P(T)$,
and $\pi_C$ and $\pi_P$ are coalgebra homomorphisms.
It follows that $D$ is a subcoalgebra of $C$ and $Q$ is a subcoalgebra of $P$.
By \leref{lem3.2}, $C\ltimes_dP$ is a $C$-Coalgebra and $\pi_P:  C\ltimes_dP\ra P$
is a $C$-Coalgebra homomorphism. Since $T$ is a subcoalgebra of $C\ltimes_dP$,
$T$ becomes a $C$-subCoalgebra of  $C\ltimes_dP$. Hence $Q=\pi_P(T)$
is a $C$-subCoalgebra of $P$, the restriction map $\pi_P: T\ra  Q$ is a $C$-Coalgebra homomorphism.

Since $T$ is a subcoalgebra and $(0,P)$ is a coideal of $C\ltimes_d P$,
$(0,R)=(0,P)\cap T$ is a coideal of $T$, and so $R=\pi_P(0,R)$ is a coideal of $Q=\pi_P(T)$.

Now we have $\rho_l(T)=(\pi_C\ot 1)\D(T)\ss (\pi_C\ot 1)(T\ot T)=D\ot T$, and similarly
$\rho_r(T)\ss T\ot D$. Since $\pi_P$ is a $C$-bicomodule homomorphism,
$\rho^P_l(Q)=\rho^P_l(\pi_P(T))=(1\ot\pi_P)(\rho_l(T))\ss(1\ot\pi_P)(D\ot T)=D\ot Q$,
and similarly $\rho^P_r(Q)\ss Q\ot D$.

Since $(0,R)$ is a coideal of $T$, $\D(0,R)\ss T\ot(0,R)+(0,R)\ot T$.
Hence $\rho_l(0,R)=(\pi_C\ot 1)\D(0,R)\ss
(\pi_C\ot 1)(T\ot(0,R)+(0,R)\ot T)=D\ot(0,R)$,
and similarly
$\rho_r(0,R)\ss(0,R)\ot D$. So
$(0,R)$  is a $C$-subbicomodule of $T$.
Since $\pi_P$ is a $C$-bicomodule homomorphism,
$R=\pi_P(0,R)$ is $C$-subbicomodules of $P$.
This proves (a) and (b).

Let $c\in D$ (resp., $p\in Q$). Then there exists some $p\in Q$ (resp., $c\in D$)
such that $(c,p)\in T$, and hence
$\eta(c)=p+R=\pi(p)$. Since $T$ is a subcoalgebra of $C\ltimes_dP$,
there exist elements $(c'_i, p'_i), (c''_i, p''_i)\in T$, $1\<i\<n$,
such that $\D(c,p)=\sum_{i=1}^n (c_i',p_i')\ot (c_i'',p_i'')$. Moreover, $\eta(c'_i)=\pi(p'_i)$
and $\eta(c''_i)=\pi(p''_i)$, $1\<i\<n$. On the other hand, we have
$\D(c, p)=\sum(c_1, 0)\ot(c_2, 0)+\sum(p_{(-1)}, 0)\ot(0, p_{(0)})+\sum(0, p_{(0)})\ot(p_{(1)}, 0)+\sum(0, p_1)\ot(0, p_2)$.
Hence
$$\begin{array}{rl}
\sum_{i=1}^n (c_i',p_i')\ot (c_i'',p_i'')=&\sum(c_1, 0)\ot(c_2, 0)+\sum(p_{(-1)}, 0)\ot(0, p_{(0)})\\
&+\sum(0, p_{(0)})\ot(p_{(1)}, 0)+\sum(0, p_1)\ot(0, p_2).\\
\end{array}$$
Applying $\pi_C\ot\pi_C$, $\pi_C\ot\pi_P$, $\pi_P\ot\pi_C$ and $\pi_P\ot\pi_P$ to the above equation respectively, one gets the following equations:
$$\begin{array}{ll}
\sum_{i=1}^n c_i'\ot c_i''=\sum c_1\ot c_2,
&\sum_{i=1}^n c_i'\ot p_i''=\sum p_{(-1)}\ot p_{(0)},\\
\sum_{i=1}^n p_i'\ot c_i''=\sum p_{(0)}\ot p_{(1)},
&\sum_{i=1}^n p_i'\ot p_i''=\sum p_1\ot p_2.\\
\end{array}$$
Hence $\sum \eta(c)_{(-1)}\ot\eta(c)_{(0)}=\sum \pi(p)_{(-1)}\ot\pi(p)_{(0)}=\sum p_{(-1)}\ot\pi(p_{(0)})
=\sum_{i=1}^n c_i'\ot \pi(p_i'')=\sum_{i=1}^n c_i'\ot\eta( c_i'')=\sum c_1\ot\eta(c_2)$.
Similarly, one can check that $\sum\eta(c)_{(0)}\ot\eta(c)_{(1)}=\sum \eta(c_1)\ot c_2$.
Therefore, $\eta$ is a $C$-bicomodule homomorphism.
We also have
$\sum\eta(p_{(-1)})\ot p_{(0)} =\sum_{i=1}^n\eta(c_i')\ot p_i'' =\sum_{i=1}^n\pi(p_i')\ot p_i''=\sum\pi(p_1)\ot p_2$,
and similarly $\sum p_{(0)}\ot\eta(p_{(1)})=\sum p_1\ot\pi(p_2)$. This shows (c).

Conversely, assume that (a), (b) and (c) are satisfied.
Let $(c,p)\in T$. Then $c\in D$, $p\in Q$ and $\eta(c)=p+R=\pi(p)\in Q/R$.
By (a) and (b), $\sum c_1\ot c_2\in D\ot D$, $\sum p_1\ot p_2\in Q\ot Q$,
$\sum p_{(-1)}\ot p_{(0)}\in D\ot Q$ and $\sum p_{(0)}\ot p_{(1)}\in Q\ot D$.
By (c), we have
$\sum c_1\ot \eta(c_2)=\sum p_{(-1)}\ot\pi(p_{(0)})$,
$\sum \eta(c_1)\ot c_2=\sum \pi(p_{(0)})\ot p_{(1)}$,
$\sum\eta(p_{(-1)})\ot p_{(0)} =\sum\pi(p_1)\ot p_2$
and $\sum p_{(0)}\ot\eta(p_{(1)})=\sum p_1\ot\pi(p_2)$.
Thus, we have
\begin{equation*}
\begin{split}
(\pi\ot 1)\D(c,p)&=(\pi\ot 1)(\sum (c_1,0)\ot (c_2,0)+\sum (p_{(-1)},0)\ot (0,p_{(0)})\\
& \quad+\sum (0,p_{(0)})\ot (p_{(1)},0)+ \sum (0,p_1)\ot (0,p_2))\\
& =\sum (c_1,\pi(0))\ot (c_2,0)
 +\sum (p_{(-1)},\pi(0))\ot (0,p_{(0)})\\
 & \quad +\sum (0,\pi(p_{(0)}))\ot(p_{(1)},0)+
 \sum (0,\pi(p_1))\ot (0,p_2)\\
 & =\sum (c_1,\pi(0))\ot (c_2,0)
 +\sum (p_{(-1)},\pi(0))\ot (0,p_{(0)})\\
 & \quad +\sum (0,\eta(c_1))\ot(c_2,0)+
 \sum (0,\eta(p_{(-1)}))\ot (0,p_{(0)})\\
 & =\sum (c_1,\eta(c_1))\ot (c_2,0)
 +\sum (p_{(-1)},\eta(p_{(-1)}))\ot (0,p_{(0)}).\\
\end{split}
\end{equation*}
This implies $(\pi\ot 1)\D(c,p)\in(\pi\ot 1)(T\ot(C\ltimes_dP))$.
Note that the kernel of $\pi\ot 1:  (C\ltimes_dP)\ot(C\ltimes_dP)\ra  (C\ltimes_dP)/(0,R)\ot(C\ltimes_dP)
= (C\ltimes_d(P/R))\ot(C\ltimes_dP)$ is $(0,R)\ot (C\ltimes_dP)$.
Since $(0,R)\ot (C\ltimes_dP)\subseteq T\ot(C\ltimes_dP)$, $\D(c,p)\in T\ot(C\ltimes_dP)$.
Similarly, one can show that $\D(c,p)\in (C\ltimes_dP)\ot T$. Hence
$\D(c,p)\in(T\ot(C\ltimes_dP))\cap((C\ltimes_dP)\ot T)=T\ot T$, and so
$T$ is a subcoalgebra of $C\ltimes_dP$.

Furthermore, we have $\sum\eta(c_1)\ot\eta(c_2)=(\eta\ot 1)(\sum c_1\ot\eta(c_2))
=(\eta\ot 1)(\sum p_{(-1)}\ot\pi(p_{(0)}))=(1\ot\pi)(\sum \eta(p_{(-1)})\ot p_{(0)})
=(1\ot\pi)(\sum \pi(p_1)\ot p_2)=\sum \pi(p_1)\ot \pi(p_2)=\sum \pi(p)_1\ot \pi(p)_2=\sum \eta(c)_1\ot \eta(c)_2$.
Hence $\eta$ is a coalgebra homomorphism.
\end{proof}

By \prref{prop3.3}, for any subspaces $D\ss C$ and $R\ss Q\ss P$, if there exists a linear surjection $\eta: D\ra Q/R$,
such that (a), (b) and (c) in \prref{prop3.3} are satisfied, then the following set is a subcoalgebra of $C\ltimes_d P$:
$$\{(c,p)\in C\ltimes_d P|c\in D, p\in Q, \eta(c)=p+R\}.$$
Conversely, any subcoalgebra of $C\ltimes_d P$ has this form.

If $T$ is a subcoalgebra of $C\ltimes_d P$, then ${\rm Ker}(\eta)=E$ is a subcoalgebra,
and hence a coideal  of $C$ (or $D$).
Thus, one gets the following corollary.

\begin{corollary}\colabel{coro3.4}
	$T$ is a subcoalgebra of $C\ltimes_dP$ if and only if the following hold:
	\begin{enumerate}
		\item[(a)] $D$ and $E$ are subcoalgebras of $C$.
		\item[(b)] $Q$ is a $C$-subCoalgebra of $P$ and $R$ a $C$-Coideal of $Q$, and $\rho^P_l(Q)\subseteq D\ot Q$
		and $\rho^P_r(Q)\subseteq Q\ot D$, where $\rho^P_l$ and $\rho^P_r$ are the $C$-comodule structure maps of $P$.
		Hence $Q$ and $Q/R$ are both $D$-Coalgebras.
		\item[(c)] $\overline{\eta}: D/E\ra Q/R$ is a $C$-bicomodule (or $D$-bicomodule) isomorphism, and
		$$\begin{array}{c}
		\sum\overline{\eta}(\overline{p_{(-1)}})\ot p_{(0)} =\sum\pi(p_1)\ot p_2,\
		\sum p_{(0)}\ot\overline{\eta}(\overline{p_{(1)}})=\sum p_1\ot\pi(p_2),\ \forall p\in Q,
		\end{array}$$
		where $\overline{x}$ denotes the image of $x\in D$ under the canonical projection $D\ra D/E$.
	\end{enumerate}
	If this is the case, then $\overline{\eta}$ is a coalgebra isomorphism, and hence it is
	a $C$-Coalgebra (or $D$-Coalgebra) isomorphism.
\end{corollary}

\begin{proof}
It follows from \prref{prop3.3} and ${\rm Ker}(\eta)=E$.
\end{proof}	

\begin{corollary}\colabel{coro3.5}
If $T$ is a subcoalgebra of $C\ltimes_dP$, then $(E,R)$ is a coideal of $T$ and
$$T/(E,R)\cong D/E \cong Q/R$$
as coalgebras.
\end{corollary}
\begin{proof}
Suppose that $T$ is a subcoalgebra of $C\ltimes_dP$.
By \coref{coro3.4}, $E$ and $D$ are subcoalgebras of $C$, $Q$ is a subcoalgebra of $P$
and $R$ is a coideal of $Q$. Moreover, $D/E \cong Q/R$ as coalgebras.
Hence the canonical
projection $D\ra  D/E$, $d\mapsto d+E$ is a coalgebra epimorphism. Then by \leref{lem3.1}, the composition
$\theta: T\xrightarrow{\pi_C}D\ra D/E$, $(c, p)\mapsto c+E$ is a coalgebra epimorphism.
Let $(c,p)\in T$. Then $c\in D$, $p\in Q$, $\eta(c)=p+R\in Q/R$ and $\theta(c,p)=c+E\in D/E$.
Since ${\rm Ker}(\eta)=E$, $(c,p)\in{\rm Ker}(\theta)$ iff $c\in E$  iff $\eta(c)=0$ iff $p\in R$.
Hence ${\rm Ker}(\theta)=\{(c,p)|c\in E, p\in R\}=(E,R)$.
It follows that $(E,R)$ is a coideal of $T$ and $T/(E,R)\cong D/E$
as coalgebras.
\end{proof}

\section{some applications}\selabel{sec4}
\subsection{The ideals of $k\ltimes_d I$}
Let $I$ be  an algebra without identity. Then we can
construct an algebra Dorroh extension $k\ltimes_d I$, which is a $k$-algebra with the identity $(1,0)$.

Let $K$ be an ideal of $k\ltimes_d I$. Then $K=\{(\a,-x)\in k\ltimes_d I|\a\in B, x\in J , \varphi(x)=\a+Z\}$ by \prref{prop2.1}, where $Z\ss B\ss k$ and $L\ss J\ss I$ are subspaces, and $\varphi: J\ra B/Z$ is a linear surjection with ${\rm Ker}(\varphi)=L$ such that (a), (b) and (c) in \prref{prop2.1} are satisfied.
Clearly, $B=Z=0$, or $B=Z=k$, or $B=k$ and $Z=0$.

If $B=Z=0$, then $J=L$ is an ideal of $I$ and $K=(0, L)=(0, J)$.
If $B=Z=k$, then $\varphi(0)=1+Z$. Hence $L=I$ by \prref{prop2.1}(c),
and so $J=I$. Thus, $K=k\ltimes_d I$. If $B=k$ and $Z=0$, then $J$ is an subalgebra of $I$, $\varphi: J\ra k$ is an algebra epimorphism such that
$\varphi(x)=\a$ implies $\a y-xy, \a y-yx\in{\rm Ker}(\varphi)$ for any $y\in I$.
Moreover, $K=\{(\a, -x)| \a\in k, x\in J, \varphi(x)=\a\}$.

Summarizing the above discussion, we have the following proposition.

\begin{proposition}\prlabel{prop4.1}
Let $K$ be a subspace of $k\ltimes_d I$. Then $K$ is an ideal of $k\ltimes_d I$
if and only if one of the following hold.
\begin{enumerate}
\item[(a)] $K=(0,L)$, where $L$ is an ideal of $I$.
\item[(b)] $K=k\ltimes_d I$.
\item[(c)] $K=\{(\a, -x)|\a\in k, x\in J, \varphi(x)=\a\}$,
where $J$ is an subalgebra of $I$, $\varphi: J\ra k$ is an algebra epimorphism, and if $\varphi(x)=\a$ then $\varphi(\a y-xy)=\varphi(\a y-yx)=0$ for any $y\in I$.
\end{enumerate}
\end{proposition}

By \prref{prop4.1}, a nontrivial ideal of $k\ltimes_d I$ may not be contained in $(0, I)$. In the following, we give such an example.

\begin{example}
Let $U$ be an algebra without identity. Then $I=k\times U$ is also an algebra without identity.
Consider the Dorroh extension $k\ltimes_d I=k\ltimes_d (k\times U)$.
Let $J=I$. Define $\varphi: J\ra k$ by $\varphi(\a,u)=\a$. Then $\varphi$ is an algebra homomorphism.
If $\varphi(x)=\a$, then $x=(\a, u)\in J$ for some $u\in U$. In this case, for any $y=(\b,v)\in I=k\times U$,
$\a y-xy=(\a\b, \a v)-(\a\b, uv)=(0, \a v-uv)$ and  $\a y-yx=(\a\b, \a v)-(\b\a, vu)=(0, \a v-vu)$,
and hence $\varphi(\a y-xy)=\varphi(\a y-yx)=0$. Thus, by \prref{prop4.1}(c), the following set
is an ideal of $k\ltimes_d I=k\ltimes_d (k\times U)$:
$$K=\{(\a, -x)|\a\in k, x=(\a, u)\in J\}=\{(\a, -\a, u)|\a\in k, u\in U\}.$$
\end{example}

\subsection{The ideals of trivial algebra extensions}
Let $A$ be a unital algebra and $M$ a unital $A$-bimodule.
Then by \cite[Example 1.3(c)]{YC}, the trivial algebra extension $A\ltimes M$
is also an algebra Dorroh extension of $A$.
Let $K$ be a subspace of $A\ltimes M$. Then by the discussion in \seref{sec2},
there are subspaces $Z\ss B\ss A$ and $J\ss M$,
and a linear surjection $\varphi: J\ra B/Z$ such that
$K=\{(a,-x)\in A\ltimes M|a\in B, x\in J , \varphi(x)=\a+Z\}$.
With these notations, we have the following corollary.

\begin{corollary}\colabel{coro4.3}
$K$ is an ideal of $A\ltimes M$ if and only if the following hold.
\begin{enumerate}
\item[(a)]  $B$ and $Z$ are ideals of $A$, and $B^2\ss Z$.
\item[(b)] $J$ is an $A$-subbimodules of $M$.
\item[(c)] $\varphi$ is an $A$-bimodule epimorphism, and $BM\ss\mathrm{Ker}(\varphi)$ and $MB\ss \mathrm{Ker}(\varphi)$.
\end{enumerate}
\end{corollary}

\begin{proof}
By \cite[Example 1.3(c)]{YC}, $M$ is an algebra with the multiplication given by $xy=0$ for any $x, y\in M$.
Hence $M^2=0$, $ay-xy=ay$ and $ya-yx=ya$ for any $a\in A$ and $x, y\in M$.

If $K$ is an ideal of $A\ltimes M$, then (a), (b), (c) in \prref{prop2.1} are satisfied.
Hence $B$ and $Z$ are ideals of $A$, $J$ is an $A$-subAlgebra of $M$
and $\varphi: J\ra B/Z$ is an $A$-Algebra homomorphism. Moreover, if $\varphi(x)=a+Z$
then $ay, ya\in{\rm Ker}(\varphi)$ for any $y\in M$ since $ay-xy=ay$ and $ya-yx=ya$.
Hence $BM\ss\mathrm{Ker}(\varphi)$ and $MB\ss \mathrm{Ker}(\varphi)$ since $\varphi$ is surjective.
By $M^2=0$,  $J^2=0$. Since $\varphi$ is an algebra epimorphism, $(B/Z)^2=0$, i.e., $B^2\ss Z$.
Thus, (a), (b) and (c) hold. Conversely, if (a), (b) and (c) hold, then  $K$ is an ideal of $A\ltimes M$ by \prref{prop2.1}.
\end{proof}

%

\subsection{The subcoalgebras of $k\ltimes_d P$}
Let $P$ be a coalgebra without counit. By \cite[Example 2.11(a)]{YC},
there is a coalgebra Dorroh extension $k\ltimes_d P$, which is a coalgebra
with the counit $\ep$ given by $\ep(\a, p)=\a$, $\forall \a\in k, p\in P$.

Let $T$ be a subcoalgebra of $k\ltimes_d P$. Then
$$T=\{(\a, p)\in k\ltimes_d P|\a\in D, p\in Q, \eta(\a)=p+R\},$$
where $E\ss D\ss k$ and $R\ss Q\ss P$ are subspaces,
and $\eta: D\ra Q/R$ is a linear surjection with ${\rm Ker}(\eta)=E$ such that (a), (b) and (c) in \prref{prop3.3}
are satisfied. Clearly, $D=E=0$, or $D=E=k$, or $D=k$ and $E=0$.

If $D=E=0$, then $Q=R$ since $\eta$ is surjective. By \prref{prop3.3}(b),  $\rho^P_l(Q)\ss D\ot Q=0$.
However, $\rho^P_l(p)=1\ot p\in k\ot Q$ for any $p\in Q$. Hence $Q=0$, and so $T=0$.

If $D=E=k$, then $Q=R$ since $\eta$ is surjective and ${\rm Ker}(\eta)=E$.
By \prref{prop3.3}(b), $Q$ is a subcoalgebra of $P$.
Hence $T=\{(\a,p)|\a\in k, p\in Q\}$ for some subcoalgebra $Q$ of $P$.

If $D=k$ and $E=0$, then dim$(Q/R)=1$ since $\eta$ is surjective and ${\rm Ker}(\eta)=E$.
By \prref{prop3.3}(b) and (c), $Q$ is a subcoalgebra of $P$, $R$ is a coideal of $Q$ (or $P$),
and $\eta$ is a linear isomorphism. Moreover, $\eta(1)\ot p=\sum\pi(p_1)\ot p_2$
and $p\ot\eta(1)=\sum p_1\ot\pi(p_2)$ for any $p\in Q$, where
$\pi: Q\ra Q/R$ is the canonical projection. Choose an element $x\in Q$ such that
$\eta(1)=\pi(x)=x+R$. Then $\eta(1)\ot p=\sum\pi(p_1)\ot p_2$
and $p\ot\eta(1)=\sum p_1\ot\pi(p_2)$ are equivalent to $\D(p)-x\ot p\in R\ot Q$
and $\D(p)-p\ot x\in Q\ot R$, respectively, where $p\in Q$.
When $p=x$, the two equations are both equivalent to $x\ot x-\sum x_1\ot x_2\in R\ot R$.

\begin{proposition}\prlabel{prop4.4}
Let $T$ be a subspace of $k\ltimes_dP$. Then $T$ is a subcoalgebra of $k\ltimes_d P$ if and only if
one of the following hold.
\begin{enumerate}
\item[(a)]  $T=0$.
\item[(b)] $T=(k,Q)$, where $Q$ is a subcoalgebra of $P$.
\item[(c)] $T=\{(\a, \a x+p)| \a\in k, p\in R\}=k(1,x)+(0,R)$, where $R$ is a coideal of $P$ and $x\in P\setminus R$
such that $\D(x)-x\ot x\in R\ot R$, $\D(p)-x\ot p\in R\ot(kx+R)$ and $\D(p)-p\ot x\in(kx+R)\ot R$
for any $p\in R$.
\end{enumerate}
\end{proposition}

\begin{proof}
Suppose that $R$ is a coideal of $P$ and $x\in P\setminus R$
such that $\D(x)-x\ot x\in R\ot R$, $\D(p)-x\ot p\in R\ot(kx+R)$ and $\D(p)-p\ot x\in(kx+R)\ot R$
for any $p\in R$. Let $Q=kx+R$. Then dim$(Q/R)=1$ and $\D(x), \D(p)\in Q\ot Q$ for any $p\in R$.
Hence $\D(Q)\ss Q\ot Q$, i.e., $Q$ is a subcoalgebra of $P$. Clearly, $R$ is a coideal of $Q$.
Moreover, $\D(p)-x\ot p\in R\ot Q$ and $\D(p)-p\ot x\in Q\ot R$ for any $p\in Q$.
Thus, the proposition follows from the above discussion.
\end{proof}

\subsection{The subcoalgebras of trivial coalgebra extensions}
Let $C$ be a counital coalgebra and $M$ a counital $C$-bicomodule.
Then by \cite[Example 2.11(c)]{YC}, the trivial coalgebra extension $C\ltimes M$
is a coalgebra Dorroh extension, where $M$ is a coalgebra with $\D_M=0$.
Let $T$ be a subspace of $C\ltimes M$. Then by the discussion in \seref{sec3},
there are  subspaces $E\ss D\ss C$ and $R\ss Q\ss M$, and a linear surjection  $\eta: D\ra Q/R $
with $\mathrm{Ker}(\eta)=E$ such that $T=\{(c,m)\in C\ltimes  M|c\in D, m\in Q, \eta(c)=m+R\}$.
With these notations, we have the following corollary.

\begin{corollary}
$T$ is a subcoalgebra of $C\ltimes M$ if and only if the following hold.
\begin{enumerate}
\item[(a)]  $E$ and $D$ are subcoalgebras of $C$ and  $\D(D)\ss E\ot D+D\ot E$.
\item[(b)] $Q$ and $R$ are $C$-subbicomodules of $M$, and $\rho_l(Q)\subseteq E\ot Q$
and $\rho_r(Q)\subseteq Q\ot E$, where $\rho_l$ and $\rho_r$ are the $C$-comodule structure maps of $M$.
\item[(c)] $\eta$ is a $C$-bicomodule homomorphism.
\end{enumerate}
\end{corollary}

\begin{proof}
Note that $\D_M(m)=\sum m_1\ot m_2=0$ for any $m\in M$ as stated above.

If $T$ is a subcoalgebra of $C\ltimes M$, then (a), (b) and (c) in \prref{prop3.3} are satisfied.
Hence $E$ and $D$ are subcoalgebras of $C$, $Q$ and $R$ are $C$-subbicomodules of $M$, and $\eta: D\ra Q/R$ is a $C$-bicomodule homomorphism. Moreover, 
$\sum\eta(m_{(-1)})\ot m_{(0)} =0$ and $\sum m_{(0)}\ot\eta(m_{(1)})=0$ for any $m\in Q$.
Therefore, $\rho_l(m)\in{\rm Ker}(\eta\ot 1)=E\ot Q$
and  $\rho_r(m)\in{\rm Ker}(1\ot\eta)=Q\ot E$ for any $m\in Q$. That is,
$\rho_l(Q)\ss E\ot Q$ and $\rho_r(Q)\ss Q\ot E$.
Note that $Q/R$ is a coalgebra with the comultiplication $\D_{Q/R}=0$.
By \prref{prop3.3}, $\eta$ is a homomorphism of coalgebra (without counit).
Hence $(\eta\ot\eta)(\D(D))=\D_{Q/R}(\eta(D))=0$,
and so $\D(D)\ss{\rm Ker}(\eta\ot\eta)=E\ot D+D\ot E$. Thus, (a), (b) and (c) hold.
Conversely, if (a), (b) and (c) hold, then $T$ is a subcoalgebra of $C\ltimes M$ by \prref{prop3.3}
and the proof above.
\end{proof}

\centerline{ACKNOWLEDGMENTS}

This work is supported by NNSF of China (No. 11571298, 11971418) and
Graduate student scientific research innovation projects in Jiangsu Province, No. XKYCX18\_036


\begin{thebibliography}{99}
\bibitem{Abe}
E. Abe, Hopf Algebras, Cambridge University Press, Cambridge, 1980.
\bibitem{AJM}
I. Alhribat, P. Jara and A. I. M\'{a}rquez, General Dorroh extensions, Missouri J. Math. Sci.,
27(1) (2015): 64-70.
\bibitem{CS}
Y. Q. Chen, K. P. Shum, Quasi-direct sums of rings and their radicals, Comm. Alg., 25(9) (1997): 2043-3055.
\bibitem{CMRS}
C, Cibils, E. Marcos, M. J. Redondo and A. Solotar, The Cohomology of split algebras
and of trivial extensions, Glasgow Math. J., 45 (2003): 21-40.
\bibitem{DF}
M. D'anna, M. Fontana,  An amalgamated duplication of a ring along an ideal:
the basic properties, J.  Algebra Appl. 6(3) (2007): 443-459.
\bibitem{D}
J. L. Dorroh, Concerning adjunctions to algebras, Bull. Amer. Math. Soc., 38 (1932): 85-88.
\bibitem{F}
D. Fechete, Some categorial aspects of the Dorroh extensions,  Acta Polytechnica Hungarica,
8(4)  (2011): 149-160.
\bibitem{M}
Z. Mesyan, The ideals of an ideal extension, J. Algebra Appl. 9(3)(2010): 407-431.
\bibitem{Mo}
S. Montgomery, Hopf Algebras and Their Actions on Rings, CBMS Reg. Conf. Ser. Math.
82, Amer. Math. Soc., Providence, RI, 1993.
\bibitem{Rad}
D.E. Radford, The structure of Hopf algebras with a projection, J. Algebra, 92 (1985): 322-347.
\bibitem{Ro}
J. Rosenberg, Algebraic $K$-Theory and Its Applications, Springer-Verlag, New
York, 1994.
\bibitem{Sw}
M. Sweedler, Hopf Algebras, Benjamin, New York, 1969.
\bibitem{YC}
L. You and H. X. Chen, Dorroh extensions of algebras and coalgebras, I, arXiv: 2007.02506 [math.RA].
\bibitem{ZT}
H. Zhu and H. J. Tang, The trivial extension of coalgebra, J. of Yangzhou University (Natural Science Edition), 11(3) (2008): 1-5. (In Chinese)
\end{thebibliography}
\end{document}